\documentclass[12pt]{amsart}
\usepackage{amssymb}
\usepackage{hyperref}
\usepackage{tabularx}
\usepackage{booktabs}
\usepackage{caption}

\newtheorem{thm}{Theorem}[section]
\newtheorem{lem}[thm]{Lemma}
\newtheorem{cor}[thm]{Corollary}
\newtheorem{prop}[thm]{Proposition}
\newtheorem{ex}[thm]{Example}

\newtheorem{qu}[thm]{Question}

\newtheorem*{prob*}{Open problem}

\theoremstyle{definition}

\newtheorem{defi}[thm]{Definition}

\theoremstyle{remark}

\newtheorem{rem}[thm]{Remark}
\newtheorem*{rem*}{Remark}

\newcommand{\kringel}{\mathbin{\raise1pt\hbox{$\scriptstyle\circ$}}}
\newcommand{\pkt}{\mathbin{\raise0pt\hbox{$\scriptstyle\bullet$}}}

\newcommand{\C}{\mathbb{C}}

\newcommand{\N}{\mathbb{N}}

\newcommand{\Z}{\mathbb{Z}}

\newcommand{\tr}{\mathop{\rm tr}}
\newcommand{\ad}{{\rm ad}}
\newcommand{\Ad}{\mathop{\rm Ad}}
\newcommand{\Ann}{{\rm Ann}}
\newcommand{\End}{{\rm End}}
\newcommand{\Der}{{\rm Der}}

\newcommand{\Lg}{\mathfrak{g}}

\newcommand{\Ll}{\mathfrak{l}}
\newcommand{\Ln}{\mathfrak{n}}

\newcommand{\Lr}{\mathfrak{r}}
\newcommand{\Ls}{\mathfrak{s}}

\newcommand{\abs}[1]{\lvert#1\rvert}

\newcommand{\al}{\alpha}
\newcommand{\be}{\beta}
\newcommand{\ga}{\gamma}
\newcommand{\de}{\delta}

\newcommand{\la}{\lambda}

\newcommand{\ra}{\rightarrow}

\renewcommand{\phi}{\varphi}

\begin{document}


\title[PA-structures]{Post-Lie algebra structures for nilpotent Lie algebras}

\author[D. Burde]{Dietrich Burde}
\author[C. Ender]{Christof Ender}
\author[W. Moens]{Wolfgang Alexander Moens}
\address{Fakult\"at f\"ur Mathematik\\
Universit\"at Wien\\
  Oskar-Morgenstern-Platz 1\\
  1090 Wien \\
  Austria}
\email{dietrich.burde@univie.ac.at}
\address{Fakult\"at f\"ur Mathematik\\
Universit\"at Wien\\
  Oskar-Morgenstern-Platz 1\\
  1090 Wien \\
  Austria}
\email{christof.ender@univie.ac.at}
\address{Fakult\"at f\"ur Mathematik\\
Universit\"at Wien\\
  Oskar-Morgenstern-Platz 1\\
  1090 Wien \\
  Austria}
\email{wolfgang.moens@univie.ac.at}

\date{\today}

\subjclass[2000]{Primary 17B30, 17D25}
\keywords{Post-Lie algebra, Pre-Lie algebra, LR-algebra, PA-structure, CPA-structure}

\begin{abstract}
We study post-Lie algebra structures on $(\Lg,\Ln)$ for nilpotent Lie algebras. First we show that if $\Lg$ is 
nilpotent such that $H^0(\Lg,\Ln)=0$, then also $\Ln$ must be nilpotent, of bounded class. For post-Lie algebra structures 
$x\cdot y$ on pairs of $2$-step nilpotent Lie algebras $(\Lg,\Ln)$ we give necessary and sufficient conditions such that 
$x\circ y=\frac{1}{2}(x\cdot y+y\cdot x)$ defines a CPA-structure on $\Lg$, or on $\Ln$. As a corollary we obtain 
that every LR-structure on a Heisenberg Lie algebra of dimension $n\ge 5$ is complete. Finally we classify all 
post-Lie algebra structures on $(\Lg,\Ln)$ for $\Lg\cong \Ln\cong \Ln_3$, where $\Ln_3$ is the $3$-dimensional 
Heisenberg Lie algebra.
 \end{abstract}

\maketitle

\section{Introduction}

Post-Lie algebras and post-Lie algebra structures arise in many areas of mathematics and physics. One particular
area is differential geometry and the study of geometric structures on Lie groups. Here post-Lie algebras arise 
as a natural common generalization of pre-Lie algebras \cite{HEL,KIM,SEG,BU5,BU19,BU24} and LR-algebras \cite{BU34, BU38}, 
in the context of nil-affine actions of Lie groups. \\
On the other hand, post-Lie algebras have been introduced by Vallette \cite{VAL} in connection with the homology of 
partition posets and the study of Koszul operads. They have been studied by several authors in various
contexts, e.g., for algebraic operad triples \cite{LOD}, in connection with modified Yang-Baxter equations, 
Rota-Baxter operators, universal enveloping algebras, double Lie algebras, $R$-matrices, isospectral flows, 
Lie-Butcher series and many other topics \cite{BAI, ELM, GUB}. \\
Our work on post-Lie algebras centers around the existence question of post-Lie algebra structures for given
pairs of Lie algebras, on algebraic structure results, and on the classification of post-Lie algebra structures. 
For a survey on the results and open questions see \cite{BU33,BU41,BU44}. A particular interesting class of post-Lie algebra 
structures is given by {\em commutative} structures, so-called {\em CPA-structures}. For the existence question of
CPA-structures on semisimple, perfect and complete Lie algebras, see \cite{BU51,BU52}. For nilpotent Lie algebras, these
questions are usually harder to answer. In \cite{BU57} we proved, among other things, that every CPA-structure on a 
nilpotent Lie algebra without abelian factor is {\em complete}, i.e., that all left multiplications $L(x)$ are nilpotent. 
It is a natural question to ask how this result extends to general post-Lie algebra structures on pairs of nilpotent
Lie algebras. In some cases we can associate a CPA-structure on $\Lg$ or on $\Ln$ to a given PA-structure on $(\Lg,\Ln)$,
and we can show the nilpotency of the left multiplications. \\[0.2cm]
The paper is structured as follows. In section $2$ we recall the basic notions of post-Lie algebra structures, or
{\it PA-structures}, and we introduce annihilators, which generalize the ones from the case of CPA-structures.
In particular, we consider the invariant $H^0(\Lg,\Ln)$ for the $\Lg$-module $\Ln$ with the action given by
a given PA-structure. In section $3$ we prove that, given a PA-structure $x\cdot y$ on $(\Lg,\Ln)$ where $\Lg$
is nilpotent and $H^0(\Lg,\Ln)=0$, that $\Ln$ must be nilpotent of class at most $\abs{X}^{2^{\abs{X}}}$. Here $X$
is a certain finite set arising from a group grading of $\Ln$. This improves a structure result 
from \cite{BU41}, where we had shown that $\Ln$ must be solvable, without the assumption on the invariants. The proof uses 
recent results on arithmetically-free group gradings of Lie algebras, given in \cite{MOE1, MOE2}. In section $4$ we associate
to any PA-structure on pairs  $(\Lg,\Ln)$ of two-step nilpotent Lie algebras a CPA-structure on $\Lg$ or on
$\Ln$, by the formula
\[
x\circ y=\frac{1}{2}(x\cdot y+y\cdot x).
\]
However, this does not work in general. It turns out that certain identities have to be satisfied. We
determine these identities. In some special cases this also implies that all left multiplications $L(x)$ 
of the PA-structure are nilpotent, because this is true for the associated CPA-structure. This is true in particular
for $\Lg$ abelian and $\Ln$ a Heisenberg Lie algebra of dimension $n\ge 5$. \\[0.2cm]
Finally, in section $5$, we classify all PA-structures $x\cdot y$ on pairs of $3$-dimensional Heisenberg Lie algebras.
The result is a long list, with rather complicated structures. They satisfy, however, very nice properties, which
we cannot prove without the classification. For example, all left multiplications $L(x)$ are nilpotent and 
$L([x,y])+R([x,y])=0$ for all $x,y\in V$. Furthermore, $x\circ y=\frac{1}{2}(x\cdot y+y\cdot x)$ defines
a CPA-structure on $\Lg$.

\section{Preliminaries}

Let $K$ denote a field of characteristic zero. We recall the definition of a post-Lie algebra structure 
on a pair of Lie algebras $(\Lg,\Ln)$ over $K$, see \cite{BU41}:

\begin{defi}\label{pls}
Let $\Lg=(V, [\, ,])$ and $\Ln=(V, \{\, ,\})$ be two Lie brackets on a vector space $V$ over 
$K$. A {\it post-Lie algebra structure}, or {\em PA-structure} on the pair $(\Lg,\Ln)$ is a 
$K$-bilinear product $x\cdot y$ satisfying the identities:
\begin{align}
x\cdot y -y\cdot x & = [x,y]-\{x,y\} \label{post1}\\
[x,y]\cdot z & = x\cdot (y\cdot z) -y\cdot (x\cdot z) \label{post2}\\
x\cdot \{y,z\} & = \{x\cdot y,z\}+\{y,x\cdot z\} \label{post3}
\end{align}
for all $x,y,z \in V$.
\end{defi}

Define by  $L(x)(y)=x\cdot y$ and $R(x)(y)=y\cdot x$ the left respectively right multiplication 
operators of the algebra $A=(V,\cdot)$. By \eqref{post3}, all $L(x)$ are derivations of the Lie 
algebra $(V,\{,\})$. Moreover, by \eqref{post2}, the left multiplication
\[
L\colon \Lg\ra \Der(\Ln)\subseteq \End (V),\; x\mapsto L(x)
\]
is a linear representation of $\Lg$. The right multiplication $R\colon V\ra V,\; x\mapsto R(x)$
is a linear map, but in general not a Lie algebra representation. \\
If $\Ln$ is abelian, then a post-Lie algebra structure on $(\Lg,\Ln)$ corresponds to
a {\it pre-Lie algebra structure} on $\Lg$. In other words, if $\{x,y\}=0$ for all $x,y\in V$, then 
the conditions reduce to
\begin{align*}
x\cdot y-y\cdot x & = [x,y], \\
[x,y]\cdot z & = x\cdot (y\cdot z)-y\cdot (x\cdot z),
\end{align*}
i.e., $x\cdot y$ is a {\it pre-Lie algebra structure} on the Lie algebra $\Lg$, see \cite{BU41}. 
If $\Lg$ is abelian, then the conditions reduce to
\begin{align*}
x\cdot y-y\cdot x & = -\{x,y\} \\
x\cdot (y\cdot z)& = y\cdot (x\cdot z), \\
x\cdot \{y,z\} & = \{x\cdot y,z\}+\{y,x\cdot z\},
\end{align*}
i.e., $-x\cdot y$ is an {\it LR-structure} on the Lie algebra $\Ln$, see \cite{BU41}. \\[0.2cm]
Another particular case of a post-Lie algebra structure arises 
if the algebra $A=(V,\cdot)$ is {\it commutative}, i.e., if $x\cdot y=y\cdot x$ is satisfied for all 
$x,y\in V$, so that we have $L(x)=R(x)$ for all $x\in V$. Then the two Lie brackets $[x,y]=\{x,y\}$ 
coincide, and we obtain a commutative algebra structure on $V$ associated with only one Lie 
algebra \cite{BU51}:

\begin{defi}\label{cpa}
A {\it commutative post-Lie algebra structure}, or {\em CPA-structure} on a Lie algebra $\Lg$ 
is a $K$-bilinear product $x\cdot y$ satisfying the identities:
\begin{align}
x\cdot y & =y\cdot x \label{com4}\\
[x,y]\cdot z & = x\cdot (y\cdot z) -y\cdot (x\cdot z)\label{com5} \\
x\cdot [y,z] & = [x\cdot y,z]+[y,x\cdot z] \label{com6}
\end{align}
for all $x,y,z \in V$. 
\end{defi}

In  \cite{BU52}, Definition $2.5$ we had introduced the notion of an annihilator in $A$ for a CPA-structure.
This can be generalized to PA-structures as follows.

\begin{defi}
Let $A=(V,\cdot)$ be a post-Lie algebra structure on a pair of Lie algebras $(\Lg,\Ln)$.
The left and right annihilators in $A$ are defined by
\begin{align*}
\Ann_L(A) & = \{x\in A\mid x\cdot A=0\},\\
\Ann_R(A) & = \{x\in A\mid A\cdot x=0\}.
\end{align*}
\end{defi}

Both spaces are in general neither left nor right ideals of $A$, unlike in the case of
CPA-structures. So we view them usually just as vector subspaces of $V$. However, the next lemma
shows that the annihilators satisfy some other properties. Recall that $\Ln$ is a $\Lg$-module via 
the product $x\cdot y$ for $x\in \Lg$ and $y\in \Ln$. The zeroth Lie algebra
cohomology is given by
\[
H^0(\Lg,\Ln)=\{y\in \Ln \mid x\cdot y=0 \;\forall\, x\in \Lg\}.
\] 

\begin{lem}
The annihilators in $A$ equal the kernels of $L$ respectively $R$, i.e.,
\begin{align*}
\Ann_L(A) & =\ker(L)= \{x\in A\mid L(x)=0\},\\
\Ann_R(A) & =\ker(R)= \{x\in A\mid R(x)=0\}.
\end{align*}
The subspace $\Ann_L(A)$ is a Lie ideal of $\Lg$, and the subspace $\Ann_R(A)$ coincides with
$H^0(\Lg,\Ln)$.
\end{lem}

\begin{proof}
The equalities are obvious. Since $L\colon \Lg\ra \Der(\Ln)$ is a Lie algebra representation,
$\ker(L)$ is a Lie ideal of $\Lg$.
\end{proof}

Suppose that $V$ is $2$-dimensional, with $\Lg$ abelian and $\Ln$ non-abelian. Then there is
a basis $(e_1,e_2)$ of $V$ such that $[e_1,e_2]=0$ and $\{e_1,e_2\}=e_1$. We have classified all
PA-structures on $(\Lg,\Ln)$ in \cite{BU41}, section $3$.

\begin{ex}\label{2.5}
Every PA-structure on $(\Lg,\Ln)\cong (K^2,\Lr_2(K))$ in the above basis is of the form
\begin{align*}
e_1\cdot e_1 & = \al e_1,\quad e_2\cdot e_1=(\be+1)e_1,\\
e_1\cdot e_2 & = \be e_1,\quad e_2\cdot e_2=\ga e_1,
\end{align*}
for $\al,\be,\ga \in K$ satisfying the condition $\be(\be-1)-\al\ga=0$.
For all these PA-structures we have 
\[
\dim \Ann_L(A)=\dim \Ann_R(A)=\dim H^0(\Lg,\Ln)=1.
\]
\end{ex}

More precisely we have
\[
\Ann_L(A)=\begin{cases} \langle \ga e_1-\be e_2\rangle, \text{ if } (\be,\ga)\neq (0,0),\\
\langle e_1-\al e_2\rangle, \hspace{0.25cm} \text{ if } \be=\ga=0,
\end{cases}
\]

\[
\Ann_R(A)=\begin{cases} \langle \be e_1-\al e_2\rangle, \text{ if } (\al,\be)\neq (0,0),\\
\langle \ga e_1- e_2\rangle, \hspace{0.25cm} \text{ if } \al=\be=0.
\end{cases}
\]

\section{Nilpotency of $\Lg$ and $\Ln$}

We have proved in \cite{BU41}, Proposition $4.3$ the following structure result for post-Lie algebra
structures on $(\Lg,\Ln)$.

\begin{prop}\label{3.1}
Suppose that there exists a post-Lie algebra structure on $(\Lg,\Ln)$, where $\Lg$ is
nilpotent. Then $\Ln$ is solvable.
\end{prop}

In this section we will prove a stronger version of this proposition by applying recent results
on arithmetically-free group-gradings of Lie algebras from \cite{MOE1,MOE2}. A grading of a Lie algebra $\Ln$
by a group $(G,\circ)$ is a decomposition
\[
\Ln=\bigoplus_{g\in G} \Ln_g
\]
into homogeneous subspaces, such that for all $g,h\in G$, we have $[\Ln_g,\Ln_h]\subseteq \Ln_{g\circ h}$.
The set $X:=\{g\in G\mid \Ln_g\neq 0\}$ is called the {\it support} of the grading.
For an abelian group $(G,+)$ such a subset $X$ of $G$ is called {\it arithmetically-free}, if and only
if $X$ is finite and
\[
\{x+ky\mid k\in \N\cup \{0\}\}\subseteq X \text{ implies } y\not\in X.
\]
In general, a subset $X$ of an arbitrary group $G$ is called arithmetically-free,  if and only
if $X$ is finite and every subset of $X$ of pairwise commuting elements is arithmetically free.
The result which we want to apply is Theorem $3.14$ of \cite{MOE1} and Theorem $3.7$ of \cite{MOE2}. It it the
following result:

\begin{thm}\label{3.2}
Let $\Ln$ be a Lie algebra over a field $K$ which is graded by a group $G$. If the support $X$ of
the grading is arithmetically-free, then $\Ln$ is nilpotent of $\abs{X}$-bounded class. If $G$ is in addition
free-abelian, the bound can be given by $\abs{X}^{2^{\abs{X}}}$.
\end{thm}

What additional conditions do we need in Proposition $\ref{3.1}$, in order to conclude that $\Ln$ is
nilpotent? Certainly $\Ln$ need not be nilpotent in general, as we have seen in Example $\ref{2.5}$.
There are PA-structures on  $(\Lg,\Ln)$ for  $\Lg$ abelian and $\Lg$ solvable, but non-nilpotent.
In all these cases the space $H^0(\Lg,\Ln)$ is non-trivial. In fact, the classification of PA-structures 
in dimension $2$, given in \cite{BU41}, shows that $\Ln$ is nilpotent in all cases where $\Lg$ is nilpotent 
{\em and} $H^0(\Lg,\Ln)=0$. It turns out that this is true in general.

\begin{thm}
Suppose that there exists a post-Lie algebra structure on $(\Lg,\Ln)$, where $\Lg$ is
nilpotent and $H^0(\Lg,\Ln)=0$. Then $\Ln$ is nilpotent of class at most $\abs{X}^{2^{\abs{X}}}$.
\end{thm}

\begin{proof}
Since $\Lg$ is nilpotent there is a weight space decomposition for the $\Lg$-module $\Ln$, 
see \cite{BU57}, section $2$. It is given by
\[
\Ln=\bigoplus_{\al\in \Lg^{\ast}}\Ln_{\al},
\]
satisfying $[\Ln_{\al},\Ln_{\be}]\subseteq \Ln_{\al+\be}$ for all  $\al,\be\in \Lg^*$.
For a weight $\al$ we have $\Ln_{\al}\neq 0$, and there are only finitely many weights. 
Hence the support $X$ is finite. The grading group $G=(\Lg^{\ast},+)$ is free-abelian, so that we can 
also write
\[
\Ln=\bigoplus_{\al \in (\Z^n,+)}\Ln_{\al}.
\]
Because of $H^0(\Lg,\Ln)=0$ we know that $0$ is not a weight. Hence the support $X$ is
arithmetically-free and we can apply Theorem $\ref{3.2}$. Hence $\Ln$ is nilpotent of class at most 
$\abs{X}^{2^{\abs{X}}}$. 
\end{proof}

For PA-structures on $(\Lg,\Ln)$ where both $\Lg$ and $\Ln$ are nilpotent and indecomposable, we often see that
all left multiplication operators $L(x)$ are nilpotent. We have recently proved this in the special case
of CPA-structures, i.e., where $\Lg=\Ln$, see \cite{BU57}:

\begin{thm}
Let $x\cdot y$ be a CPA-structure on $\Lg$, where $\Lg$ is nilpotent with $Z(\Lg)\subseteq [\Lg,\Lg]$.
Then all left multiplications $L(x)$ are nilpotent.
\end{thm}

We have called a Lie algebra $\Lg$ with $Z(\Lg)\subseteq [\Lg,\Lg]$ a {\it stem Lie algebra}.
It seems that this result has a natural generalization to PA-structures on pairs of nilpotent
Lie algebras. So we pose the following question.

\begin{qu}
Let $x\cdot y$ be a PA-structure on $(\Lg,\Ln)$ where both $\Lg$ and $\Ln$ are nilpotent stem Lie
algebras. Is it true that all left multiplications $L(x)$ are nilpotent?
\end{qu}

Examples of PA-structures in low dimensions show that there are counterexamples with $\Lg$ or $\Ln$ not nilpotent.
For the following example, let $\Lg$ be the $3$-dimensional solvable non-nilpotent Lie algebra $\Lr_{3,\la}(K)$ 
with basis $\{e_1,e_2,e_3\}$ and $[e_1,e_2]=e_2$, $[e_1,e_3]=\la e_3$ for $\la\in K^{\times}$, and 
$\Ln$ be the Heisenberg Lie algebra $\Ln_3(K)$ with $\{e_1,e_2\}=e_3$. 

\begin{ex}
There is a PA-structure on $(\Lg,\Ln)$ given by
\begin{align*}
e_1\cdot e_1 & = (\la-1)e_1+\al e_2+\be e_3,\; e_1\cdot e_2 = e_2+\ga e_3,\\
e_1\cdot e_3 & = \lambda e_3,\; e_2\cdot e_1=(\ga+1)e_3,
\end{align*}
with $\al,\be,\ga\in K$, where $L(e_1)$ is not nilpotent.
\end{ex}

Indeed, $\tr L(e_1)=2\la\neq 0$, since $2\neq 0$.

\section{PA-structures on pairs of two-step nilpotent Lie algebras}

Let $(\Lg,\Ln)$ be a pair of two-step nilpotent Lie algebras and $x\cdot y$ be a PA-structure
on $(\Lg,\Ln)$. We would like to associate with $x\cdot y$ a CPA-structure on $\Lg$ or on $\Ln$,
by the formula
\[
x\circ y=\frac{1}{2}( x\cdot y+y\cdot x).
\]
This will not always give a CPA-structure. However, we can find suitable conditions on $\Lg$,
$\Ln$ and on $x\cdot y$, so that the new product indeed gives a CPA-structure. \\
Let us denote by $\ad(x)$ the adjoint 
operators for $\Lg$ with $\ad(x)(y)=[x,y]$, and by $\Ad(x)$ the adjoint operators for $\Ln$ with $\Ad(x)(y)=\{x,y\}$.
Furthermore $L(x)$ and $R(x)$ are the left and right multiplication operators. The axioms for a PA-structure
on $(\Lg,\Ln)$ in operator form are as follows:
\begin{align}
L(x)-R(x) & = \ad(x)-\Ad(x) \label{op1} \\
L([x,y]) & = [L(x),L(y)] \label{op2} \\
[L(x),\Ad(y)] & = \Ad(L(x)y) \label{op3}
\end{align}
for all $x,y\in V$.

\begin{lem}
The axioms for a PA-structure on $(\Lg,\Ln)$ imply the following operator identities.
\begin{align}
[L(x),\Ad(y)]+[\Ad(x),L(y)] & = \Ad([x,y]) -\Ad(\{x,y \}) \label{10} \\
[R(x),\ad(y)]+[\ad(x),R(y)] & = [L(x),\ad(y)]+[\ad(x),L(y)] \label{11} \\
 & + [\Ad(x),\ad(y)]+[\ad(x),\Ad(y)]-2[\ad(x),\ad(y)] \nonumber
\end{align}
for all $x,y\in V$.
\end{lem}

\begin{proof}
Using \eqref{op1} and \eqref{op3} we obtain
\begin{align*}
[L(x),\Ad(y)] & = \Ad(x\cdot y) \\
              & = \Ad([x,y]-\{x,y\}+y\cdot x) \\
              & = \Ad ([x,y])-\Ad(\{x,y\})+\Ad(y\cdot x) \\
              & = \Ad ([x,y])-\Ad(\{x,y\})+[L(y),\Ad(x)]
\end{align*}
This shows \eqref{10}. Taking Lie brackets of \eqref{op1} with $\ad(x)$ and $\ad(y)$ gives
\begin{align*}
[L(x),\ad(y)]-[R(x),\ad(y)] & = [\ad(x),\ad(y)]-[\Ad(x),\ad(y)] \\
[L(y),\ad(x)]-[R(y),\ad(x)] & = [\ad(y),\ad(x)]-[\Ad(y),\ad(x)] 
\end{align*}
The difference gives \eqref{11}.
\end{proof}

If $\Lg$ and $\Ln$ are $2$-step nilpotent, then the terms $[\ad(x),\ad(y)]$ and $\Ad(\{x,y\})$ vanish.

\begin{lem}
Suppose that $x\cdot y$ is a PA-structure on $(\Lg,\Ln)$, where $\Lg$ and $\Ln$ are $2$-step nilpotent,
and
\begin{align}
[L(x)+R(x),\ad(y)] & = \ad (x\cdot y+y\cdot x) \label{12}
\end{align}
for all $x,y\in V$. Then we have
\begin{align}
[L(x)+R(x),\ad(y)] & = [L(y)+R(y),\ad(x)] \label{13} \\
2[L(x),\ad(y)]+2[\ad(x),L(y)] & = [\ad(y),\Ad(x)]+[\Ad(y),\ad(x)] \label{14}
\end{align}
\end{lem}

\begin{proof}
Since  $\ad (x\cdot y+y\cdot x)$ is symmetric in $x$ and $y$, \eqref{12} implies \eqref{13}. 
We can rewrite it as 
\[
[R(x),\ad(y)]+[\ad(x),R(y)]= [\ad(y),L(x)]+[L(y),\ad(x)]
\]
Together with \eqref{11} we obtain \eqref{14}.
\end{proof}

\begin{prop}\label{4.3}
Let $x\cdot y$ be a PA-structure on $(\Lg,\Ln)$, where $\Lg$ and $\Ln$ are $2$-step nilpotent. Then
\[
x\circ y=\frac{1}{2}( x\cdot y+y\cdot x)
\]
defines a CPA-structure on $\Lg$ if and only if \eqref{12} holds for all $x,y\in V$.
\end{prop}

\begin{proof}
Let $\ell(x)$ and $r(x)$ be the left and right multiplications given by $\ell(x)(y)=x\circ y$
and $r(x)(y)=y\circ x$. By  \eqref{op1} we have 
\begin{align*}
\ell(x) & =\frac{1}{2}(L(x)+R(x)) \\
        & = L(x)-\frac{1}{2}\ad(x)+\frac{1}{2}\Ad(x).
\end{align*}
The axioms of a CPA-structure on $\Lg$ in operator form are given by
\begin{align*}
\ell(x) & = r(x) \\
\ell([x,y]) & =[\ell(x),\ell(y)] \\
[\ell(x),\ad(y)]& =\ad (\ell(x)y)
\end{align*}
We will show that these axioms follow from \eqref{12}. The computations will also show that the axioms
are in fact equivalent to \eqref{12}. Clearly $\ell(x)=r(x)$ is obvious since the product $x\circ y$
is commutative. The third identity is just \eqref{12} if we write $\ell(x) =\frac{1}{2}(L(x)+R(x))$.
So it remains to show the second identity. The left-hand side is given by
\begin{align*}
\ell([x,y]) & = L([x,y]) -\frac{1}{2}\ad([x,y])+\frac{1}{2}\Ad([x,y]) \\
           & = L([x,y]) +\frac{1}{2}\Ad([x,y]), 
\end{align*}
because $\Lg$ is $2$-step nilpotent. On the other hand, using $[\ad(x),\ad(y)]=[\Ad(x),\Ad(y)]=0$ we have
\begin{align*}
[\ell(x),\ell(y)] & = [ L(x)-\frac{1}{2}\ad(x)+\frac{1}{2}\Ad(x),  L(y)-\frac{1}{2}\ad(y)+\frac{1}{2}\Ad(y)] \\
 & = [L(x),L(y)]-\frac{1}{2}[L(x),\ad(y)]+\frac{1}{2}[L(x),\Ad(y)]-\frac{1}{2}[\ad(x),L(y)] \\
 & -\frac{1}{4}[\ad(x),\Ad(y)]+\frac{1}{2}[\Ad(x),L(y)]-\frac{1}{4}[\Ad(x),\ad(y)]
\end{align*}
We have $[L(x),L(y)]=L([x,y])$ by \eqref{op2} and 
\[
\frac{1}{2}\Ad([x,y])=\frac{1}{2}[L(x),\Ad(y)]+\frac{1}{2}[\Ad(x),L(y)]
\]
by \eqref{10}, because $\Ln$ is $2$-step nilpotent. For the difference we obtain 
\begin{align*}
\ell([x,y]) - [\ell(x),\ell(y)] & = \frac{1}{2}[L(x),\ad(y)]+\frac{1}{2}[\ad(x),L(y)] +\frac{1}{4}[\ad(x),\Ad(y)]
+\frac{1}{4}[\Ad(x),\ad(y)] \\
 & = 0
\end{align*}
by using \eqref{14}. 
\end{proof}

\begin{rem}\label{4.4}
The identity \eqref{12} can be rewritten as
\begin{align}
x\cdot [y,z]+[y,z]\cdot x & = [y,x\cdot z]+[y,z\cdot x]-[z,x\cdot y]-[z, y\cdot x] \label{15}
\end{align}
for all $x,y,z\in V$. This yields another operator version of \eqref{12}:
\begin{align}
L([y,z])+R([y,z]) & = \ad(y)(L(z)+R(z))-\ad(z)(L(y)+R(y)) \label{16}
\end{align} 
for all $y,z\in V$. This identity is trivially satisfied if $\Lg$ is abelian.
\end{rem}

It is quite remarkable that identity \eqref{12} holds for all PA-structures on $(\Lg,\Ln)$, where $\Lg$ and $\Ln$ 
are isomorphic to the $3$-dimensional Heisenberg Lie algebra, see Corollary $\ref{5.3}$. However, this
is not always true. Let $(e_1,\ldots ,e_5)$ be a basis of $V$ and define the Lie brackets of $\Lg$ and $\Ln$ by
\begin{align*}
[e_1,e_2] & = e_5,\; [e_3,e_4] = e_5,\\
\{e_1,e_4\} & = e_5,\; \{e_2,e_3\}=e_5
\end{align*}
Then $\Lg$ and $\Ln$ are both isomorphic to the $5$-dimensional Heisenberg Lie algebra. 
\begin{ex}
There exists a PA-structure on the above pair $(\Lg,\Ln)$, which does not satisfy the
identity \eqref{12}. It is given by
\begin{align*}
e_2\cdot e_1 & = -e_5,\; e_3\cdot e_2=e_5,\; e_3\cdot e_3=e_2,\\
e_4\cdot e_1 & = e_5,\; e_4\cdot e_3=-e_5.
\end{align*}
Hence we cannot apply Proposition $\ref{4.3}$.
\end{ex}
Indeed, setting $(x,y,z)=(e_3,e_1,e_3)$ in \eqref{15} we obtain 
\begin{align*}
0 & = [e_1,2e_3\cdot e_3]=2e_5, 
\end{align*}
a contradiction. \\[0.2cm]
We can apply Proposition $\ref{4.3}$ to the case where $\Ln$ is abelian. In this case, PA-structures on
$(\Lg,\Ln)$ correspond to pre-Lie algebra structures on $\Lg$.

\begin{cor}
Let $x\cdot y$ be a pre-Lie algebra structure on $\Lg$, where $\Lg$ is $2$-step nilpotent.
Then
\[
x\circ y=\frac{1}{2}( x\cdot y+y\cdot x)
\]
defines a CPA-structure on $\Lg$ if and only if all $L(x)$ are derivations of $\Lg$.
If in addition $Z(\Lg)\subseteq [\Lg,\Lg]$, then all $L(x)$ are nilpotent.
\end{cor}

\begin{proof}
We have $R(x)=L(x)-\ad(x)$ so that identity \eqref{12} reduces to
\begin{align*}
[L(x)+R(x),\ad(y)] & = [2L(x)-\ad(x),\ad(y)]\\
                   & = 2[L(x),\ad(y)],
\end{align*}
and 
\begin{align*}
\ad(x\cdot y+y\cdot x) & = \ad (2x\cdot y-[x,y])\\
                      & = 2\ad(L(x)y).
\end{align*}
So it is equivalent to $[L(x),\ad(y)]=\ad (L(x)y)$, which says that all $L(x)$ are
derivations of $\Lg$. So $x\circ y$ is a CPA-product on $\Lg$ by Proposition $\ref{4.3}$.
With $\ell(x)=x\circ y$ we have
\[
L(x)=\ell(x)+\frac{1}{2}\ad(x).
\]
By Theorem $3.6$ of \cite{BU57} all $\ell(x)$ are nilpotent, since $Z(\Lg)\subseteq [\Lg,\Lg]$. 
Furthermore we have
\begin{align*}
[\ell(x),\ad(y)]& =\ad (\ell(x)y) \\
                & =\frac{1}{2}(\ad(x\cdot y)+\ad(y\cdot x)) \\
                & = 0,
\end{align*}
because $[y\cdot x,z]=[x\cdot y-[x,y],z]=[x\cdot y,z]$ for all $x,y,z\in \Ln$. Since
$L(x)$ is the sum of two commuting nilpotent operators, it is nilpotent.
\end{proof}

Note that Medina studied pre-Lie algebras where all $L(x)$ are derivations in \cite{MED}, under the name of 
{\it left-symmetric derivation algebras}. \\[0.2cm]
Proposition $\ref{4.3}$ has a counterpart for associated CPA-structures on $\Ln$.

\begin{prop}\label{4.6}
Let $x\cdot y$ be a PA-structure on $(\Lg,\Ln)$, where $\Lg$ and $\Ln$ are $2$-step nilpotent.
Then
\[
x\circ y=\frac{1}{2}( x\cdot y+y\cdot x)
\]
defines a CPA-structure on $\Ln$ if and only if
\begin{align}
[\ad(x),\Ad(y)] & = \Ad([x,y]) \label{17} \\
L(\{x,y\})-L([x,y]) & = \frac{1}{2}\left( \ad(\{x,y\})+[\ad(y),L(x)]+[L(y),\ad(x)]\right) \label{18}
\end{align}
for all $x,y\in V$. 
\end{prop}

\begin{proof}
Let $\ell(x)$ and $r(x)$ be the left and right multiplications given by 
$\ell(x)(y)=x\circ y$ and $r(x)(y)=y\circ x$. By \eqref{op1} we have 
\begin{align*}
\ell(x) & =\frac{1}{2}(L(x)+R(x))  \\
        & = L(x)-\frac{1}{2}\ad(x)+\frac{1}{2}\Ad(x) 
\end{align*}
The axioms of a CPA-structure on $\Ln$ are given by
\begin{align*}
\ell(x) & = r(x) \\
\ell(\{x,y\}) & =[\ell(x),\ell(y)] \\
[\ell(x),\Ad(y)]& =\Ad (\ell(x)y)
\end{align*}
The first identity is obvious. For the third identity we have
\begin{align*}
[\ell(x),\Ad(y)] & = [L(x)-\frac{1}{2}\ad(x)+\frac{1}{2}\Ad(x),\Ad(y)] \\
                 & = [L(x),\Ad(y)]-\frac{1}{2}[\ad(x),\Ad(y)]+\frac{1}{2}[\Ad(x),\Ad(y)] \\
                 & = [L(x),\Ad(y)]-\frac{1}{2}[\ad(x),\Ad(y)]
\end{align*}

and 

\begin{align*}
\Ad(\ell(x)y)  & = \Ad(L(x)y)-\frac{1}{2}\Ad([x,y])+\frac{1}{2}\Ad(\{x,y\}) \\
               & = \Ad(L(x)y)-\frac{1}{2}\Ad([x,y])
\end{align*}
By \eqref{op3} and \eqref{17} the two sides are equal. It remains to show the second identity.
We have
\begin{align*}
\ell(\{x,y\}) & = L(\{x,y\})-\frac{1}{2}\ad(\{x,y\})+\frac{1}{2}\Ad(\{x,y\}) \\
              & = L(\{x,y\})-\frac{1}{2}\ad(\{x,y\})
\end{align*}
On the other hand we have, using \eqref{10} and \eqref{17} we have
\begin{align*}
[\ad(x),\Ad(y)] & = \Ad([x,y])\\
                & = [L(x),\Ad(y)]+[\Ad(x),L(y)] \\
                & = [\Ad(x),\ad(y)] 
\end{align*}
Hence we obtain, using \eqref{op2} and $2$-step nilpotency
\begin{align*}
[\ell(x),\ell(y)] & = [L(x)-\frac{1}{2}\ad(x)+\frac{1}{2}\Ad(x),L(y)-\frac{1}{2}\ad(y)+\frac{1}{2}\Ad(y)] \\
                  & = [L(x),L(y)]-\frac{1}{2}[L(x),\ad(y)] +\frac{1}{2}[L(x),\Ad(y)]-\frac{1}{2}[\ad(x),L(y)]\\
                  & - \frac{1}{4}[\ad(x),\Ad(y)]+\frac{1}{2}[\Ad(x),L(y)]-\frac{1}{4}[\Ad(x),\ad(y)] \\
                  & = L([x,y])-\frac{1}{2}[L(x),\ad(y)]-\frac{1}{2}[\ad(x),L(y)]
\end{align*}
By \eqref{18}, both sides are equal.
\end{proof}

The identities  \eqref{17},\eqref{18} may not hold in general for PA-structures on $2$-step nilpotent Lie algebras.
Let $(e_1,e_2,e_3)$ be a basis of $V$ and define the Lie brackets of $\Lg$ and $\Ln$ by
\begin{align*}
[e_1,e_2] & = e_3,\; \{e_2,e_3\} = e_1.
\end{align*}
Then $\Lg$ and $\Ln$ are both isomorphic to the $3$-dimensional Heisenberg Lie algebra. 

\begin{ex}
There exists a PA-structure on the above pair $(\Lg,\Ln)$, which does not satisfy the
identities \eqref{17}, \eqref{18}. It is given by
\begin{align*}
e_1\cdot e_2 & = e_3,\; e_2\cdot e_3=-\frac{1}{2}e_1.
\end{align*}
Hence we cannot associate a CPA-structure on $\Ln$ to it by Proposition $\ref{4.6}$.
\end{ex}

This is the CPA-structure of type $6$ in Proposition $\ref{5.2}$ with $r_7=1$ and $\al=\be=0$. We have
$\Ad([e_2,e_3])=0$, but $[\ad(e_2),\Ad(e_3)](e_2)=\ad(e_2)\Ad(e_3)e_2=e_3$. This contradicts
\eqref{17}. Similarly, \eqref{18} does not hold for $(x,y)=(e_1,e_2)$.

\begin{cor}\label{4.8}
Let $x\cdot y$ be a PA-structure on $(\Lg,\Ln)$, where $\Lg$ is abelian, $\Ln$ is $2$-step nilpotent 
Then
\[
x\circ y=\frac{1}{2}( x\cdot y+y\cdot x)
\]
defines a CPA-structure on $\Ln$ if and only if $\{\Ln,\Ln\}\cdot \Ln=0$.
\end{cor}

\begin{proof}
If $\Lg$ is abelian then \eqref{17} is trivially satisfied and \eqref{18} reduces to $L(\{x,y\})=0$ for all $x,y\in V$.
Hence the claim follows from Proposition $\ref{4.6}$.
\end{proof}

The identity $L(\{x,y\})=0$ also implies $R(\{x,y\})=0$ by \eqref{op1} as $\ad(\{x,y\})=\Ad(\{x,y\})=0$. So we have
$\{\Ln,\Ln\}\cdot \Ln=\Ln\cdot \{\Ln,\Ln\}=0$ in the corollary. A PA-structure $x\cdot y$ on $(\Lg,\Ln)$ with $\Lg$ abelian 
corresponds to an LR-structure on $\Ln$ by $-x\cdot y$, see \cite{BU38}. So we may identify PA-structures on $(\Lg,\Ln)$
with $\Lg$ abelian with LR-structures on $\Ln$.

\begin{cor}\label{4.9}
Every LR-structure on $\Ln$, where $\Ln$ is $2$-step nilpotent with $Z(\Ln)\subseteq \{\Ln,\Ln\}$ and $\{\Ln,\Ln\}\cdot \Ln=0$ 
is complete, i.e., all $L(x)$ are nilpotent. 
\end{cor}

\begin{proof}
By Corollary $\ref{4.8}$, $x\circ y=\frac{1}{2}( x\cdot y+y\cdot x)$
defines a CPA-structure on $\Ln$. With $\ell(x)(y)=x\circ y$ we have
\[
L(x)=\ell(x)-\frac{1}{2}\Ad(x).
\]
By Theorem $3.6$ of \cite{BU57} all $\ell(x)$ are nilpotent, since $Z(\Ln)\subseteq \{\Ln,\Ln\}$. 
We have $\Ad(x)^2=0$ for all $x\in V$ and
\begin{align*}
[\ell(x),\Ad(y)]& =\Ad (\ell(x)y) \\
                & =\frac{1}{2}(\Ad(x\cdot y)+\Ad(y\cdot x)) \\
                & = 0,
\end{align*}
because $\{x\cdot y,z\}=\{x\cdot y-\{y,x\},z\}=\{x\cdot y,z\}$ for all $x,y,z\in \Ln$. Since
$L(x)$ is the difference of two commuting nilpotent operators, it is nilpotent.
\end{proof}

The following lemma is helpful to give examples of $2$-step nilpotent Lie algebras satisfying 
the conditions of Corollary $\ref{4.8}$, i.e., with
\begin{align*}
L(\{x,y\}) & = R(\{x,y\})=0
\end{align*}
for all $x,y\in V$.

\begin{lem}\label{4.10}
Let $x\cdot y$ be a PA-structure on $(\Lg,\Ln)$, where $\Lg$ is abelian and $\Ln$ is $2$-step nilpotent. Then for
each $p,q,x\in \Ln$ with $\{x,p\}=\{x,q\}=0$ we have
\[
x\cdot \{p,q\}=0.
\]
\end{lem}

\begin{proof}
By \eqref{post3} we have
\begin{align*}
0 & = q\cdot \{x,p\} =\{q\cdot x,p\}+\{x,q\cdot p\} \\
0 & = p\cdot \{x,q\} =\{p\cdot x,q\}+\{x,p\cdot q\}
\end{align*}
Using \eqref{post1}, which is $u\cdot v-v\cdot u=\{v,u\}$, and taking the difference above gives
\begin{align*}
0 & = \{q\cdot x,p\}-\{p\cdot x,q\}+\{x,q\cdot p-p\cdot q\}\\
  & = \{q\cdot x,p\}-\{p\cdot x,q\}+\{x,\{p,q\}\} \\
  & = \{q\cdot x,p\}-\{p\cdot x,q\}
\end{align*}
because $\Ln$ is $2$-step nilpotent. But $ \{q\cdot x,p\}=\{p\cdot x,q\}$  implies
\[
\{x\cdot q,p\}=\{x\cdot p,q\},
\]
because $\{v\cdot u,w\}=\{u\cdot v-\{v,u\},w\}=\{u\cdot v,w\}$ for all $u,v,w\in \Ln$. We obtain
\begin{align*}
x\cdot \{p,q\} & = \{x\cdot p,q\}+\{p,x\cdot q\}\\
               & = \{x\cdot q,p\}+\{p,x\cdot q\} \\
               & =0.
\end{align*}
\end{proof}

\begin{prop}\label{4.11}
Let $x\cdot y$ be a PA-structure on $(\Lg,\Ln)$, where $\Lg$ is abelian and $\Ln$ is a Heisenberg
Lie algebra of dimension $n\ge 5$. Then $Z(\Ln)\cdot \Ln=\Ln\cdot Z(\Ln)=0$, and
\[
x\circ y=\frac{1}{2}( x\cdot y+y\cdot x)
\]
defines a CPA-structure on $\Ln$.
\end{prop}

\begin{proof}
We may choose a basis $\{e_i,f_i,z\mid i=1,\ldots,m\}$ for $\Ln$ with Lie brackets $[e_i,f_i]=z$ for all $1\le i\le m$.
Then $\{\Ln,\Ln\}=Z(\Ln)=\langle z \rangle$. Taking $(p,q)=(e_1,f_1)$ in Lemma $\ref{4.10}$ yields
\[
x\cdot z= x\cdot \{p,q\}=0
\]
for all basis vectors $x$ of $\Ln$ different from $e_1,f_1$. Because of $m\ge 2$ we can choose $(p,q)=(e_2,f_2)$
to obtain $x\cdot z=0$ also for $x=e_1$ and $x=f_1$. We obtain $Z(\Ln)\cdot \Ln=\Ln\cdot Z(\Ln)=0$, and 
the claim follows from Corollary $\ref{4.8}$.
\end{proof}

Note that the proposition is not true for the $3$-dimensional Heisenberg Lie algebra $\Ln_3(K)$. Let
$\{e_1,e_2,e_3\}$ be a basis with $\{e_1,e_2\}=e_3$. Then 
\[
e_2\cdot e_1=e_3,\; e_2\cdot e_2=-e_2,\; e_2\cdot e_3=-e_3,e_3\cdot e_2 =-e_3
\]
is a PA-structure on $(K^3,\Ln_3(K))$, namely the negative of the LR-structure $A_4$ in \cite{BU34}, Proposition $3.1$.
We have $e_2\cdot e_3\neq 0$, so that $\Ln\cdot Z(\Ln)\neq 0$. Indeed, the argument in the above proof does not work 
for $m=1$.

\begin{cor}
Every LR-structure on $\Ln$, where $\Ln$ is a Heisenberg Lie algebra of dimension $n\ge 5$ is complete.
\end{cor}

\begin{proof}
By Proposition $\ref{4.11}$, every LR-structure on $\Ln$ satisfies $\{\Ln,\Ln\}\cdot \Ln=0$,
so that the claim follows from Corollary $\ref{4.9}$ since $Z(\Ln)=\{\Ln,\Ln\}$. 
\end{proof}

\section{PA-Structures on pairs of Heisenberg Lie algebras}

In this section we want to list all PA-structures on $(\Lg,\Ln)$ where $\Lg$ is the $3$-dimensional Heisenberg 
Lie algebra $\Ln_3(K)$ and $\Ln\cong \Lg$. There is a basis $(e_1,e_2,e_3)$ of $V$ such that $[e_1,e_2]=e_3$, and
the Lie brackets of $\Ln$ are given by
\begin{align*}
\{e_1,e_2\} & = r_1e_1+r_2e_2+r_3e_3,\\
\{e_1,e_3\} & = r_4e_1+r_5e_2+r_6e_3,\\
\{e_2,e_3\} & = r_7e_1+r_8e_2+r_9e_3,\\
\end{align*}
with structure constants $r=(r_1,\ldots ,r_9)\in K^9$. The Jacobi identity gives polynomial conditions on 
these structure constants. The Lie algebra $\Ln$ is isomorphic to the Heisenberg Lie algebra $\Ln_3(K)$ 
if and only if $\Ln$ is $2$-step nilpotent with $1$-dimensional center. 

\begin{lem}
Let $\Ln$ be isomorphic to the Heisenberg Lie algebra over $K$. Then every structure constant vector $r$ 
for $\Ln$ belongs to one of the following three types $A$, $B$ and $C$:
\begin{align*} 
r & = \left(r_1,r_2,r_3,-\frac{r_1r_2}{r_3}, -\frac{r_2^2}{r_3},-r_2,\frac{r_1^2}{r_3},\frac{r_1r_2}{r_3},r_1\right),\; 
r_3\neq 0\\[0.2cm]
r & = \left(0,0,0,r_4,r_5,0,-\frac{r_4^2}{r_5},-r_4,0\right),\; r_5\neq 0 \\[0.2cm]
r & = (0,0,0,0,0,0,r_7,0,0), \; r_7\neq 0
\end{align*}
\end{lem}

\begin{proof}
Since $\Ln$ is nilpotent we have $\tr \Ad(e_i)^k=0$ for $i,k\in \{1,2,3\}$. For $k=1$ we obtain
the linear conditions $(r_6,r_8,r_9)=(-r_2,-r_4,r_1)$, and for $k=2$ we obtain the quadratic conditions
\begin{align*}
r_2^2+r_3r_5 & = 0,\\ 
r_1^2-r_3r_7 & = 0,\\ 
r_4^2+r_5r_7 & = 0. 
\end{align*}
These conditions already imply the Jacobi identity. Assume that $r_1\neq 0$. Then the quadratic equations
imply that $r_3\neq 0$, $r_5=-\frac{r_2^2}{r_3}$, $r_7=\frac{r_1^2}{r_3}$, and $r_4^2=\left( \frac{r_1r_2}{r_3}\right)^2$.
So we obtain two cases. If $r_4=\frac{r_1r_2}{r_3}$, then the nilpotency of $\Ad(e_3)$ implies that $r_2=r_4=0$.
We obtain $r=(r_1,0,r_3,0,0,0,\frac{r_1^2}{r_3},0,r_1)$, which is of type $A$ and represents the Heisenberg Lie algebra,
with $1$-dimensional center $Z(\Ln)=\langle r_1e_1+r_3e_3\rangle$. In the other case, $r_4=-\frac{r_1r_2}{r_3}$,
and we obtain
\[
r = \left(r_1,r_2,r_3,-\frac{r_1r_2}{r_3}, -\frac{r_2^2}{r_3},-r_2,\frac{r_1^2}{r_3},\frac{r_1r_2}{r_3},r_1\right) 
\]
of type $A$, with $1$-dimensional center $Z(\Ln)=\langle r_1e_1+r_2e_2+r_3e_3\rangle$. A similar analysis also gives the
result for $r_1=0$ by distinguishing $r_2\neq 0$ and $r_2=0$. In the end it is used that $r$ is not the zero vector, 
because $\Ln$ is not abelian.
\end{proof}

In the following proposition we list all possible PA-structures on pairs of Heisenberg Lie algebras
$(\Lg,\Ln)$ as above by the left multiplication operators $L(e_1),L(e_2),L(e_3)$. Surprisingly we obtain 
$L(e_3)= -\frac{1}{2}\Ad(e_3)$ in all cases, so that we need not list $L(e_3)$. The parameters in the list are in $K$.  

\begin{prop}\label{5.2}
Every PA-structure on $(\Lg,\Ln)$ with $\Lg=\Ln_3(K)$ and $\Ln\cong \Lg$ is of one of the following list.
We always have $L(e_3)=-\frac{1}{2}\Ad (e_3)$.\\[0.2cm]
1. $\Ln$ is of type $A$ with $r = \left(r_1,r_2,r_3,-\frac{r_1r_2}{r_3}, -\frac{r_2^2}{r_3},-r_2,\frac{r_1^2}{r_3},
\frac{r_1r_2}{r_3},r_1\right),\; r_2, r_3\neq 0$ and
\[
L(e_1)=\begin{pmatrix}\frac{r_1\al}{r_2} & -\frac{r_1(2r_1\al +r_2^2)}{2r_2^2} & \frac{r_1r_2}{2r_3} \\[0.2cm]
\al & -\frac{2r_1\al +r_2^2}{2r_2} & \frac{r_2^2}{2r_3} \\[0.2cm]
\be & -\frac{2r_1\be +r_2r_3}{2r_2} & \frac{r_2}{2}
\end{pmatrix},\quad 
L(e_2)=\begin{pmatrix}\frac{r_1(r_2^2-2r_1\al)}{2r_2^2} & \frac{r_1^3\al}{r_2^3} & -\frac{r_1^2}{2r_3} \\[0.2cm]
\frac{r_2^2-2r_1\al}{2r_2} & \frac{r_1^2\al}{r_2^2} & -\frac{r_1r_2}{2r_3} \\[0.2cm]
\frac{r_2(r_3-2)-2r_1\be}{2r_2} & \frac{r_1(r_1\be+r_2)}{r_2^2} & -\frac{r_1}{2}
\end{pmatrix}
\]
2. $\Ln$ is of type $A$ with $r = \left(r_1,0,r_3,0,0,0,\frac{r_1^2}{r_3},0,r_1\right),\;r_3\neq 0$ and 
\[
L(e_1)=\begin{pmatrix}0 & -\frac{r_1}{2} & 0 \\[0.2cm]
0 & 0 & 0 \\[0.2cm]
0 & \frac{2-r_3}{2} & 0
\end{pmatrix},\quad 
L(e_2)=\begin{pmatrix}\frac{r_1}{2} & \al & -\frac{r_1^2}{2r_3} \\[0.2cm]
0 & 0 & 0 \\[0.2cm]
\frac{r_3}{2} & \be & -\frac{r_1}{2}
\end{pmatrix}
\]
3. $\Ln$ is of type $A$ with $r = \left(0,0,r_3,0,0,0,0,0,0\right),\;r_3\neq 0$ and 
\[
L(e_1)=\begin{pmatrix} \al &  -\frac{\al^2}{\be} & 0 \\[0.2cm]
\be & -\al & 0 \\[0.2cm]
\ga & \de & 0
\end{pmatrix},\quad 
L(e_2)=\begin{pmatrix}-\frac{\al^2}{\be} & \frac{\al^3}{\be^2} & 0 \\[0.2cm]
-\al & \frac{\al^2}{\be} & 0 \\[0.2cm]
r_3-1+\de & \frac{\al(\be(1-r_3)-\al\ga-2\be\de)}{\be^2}  & 0
\end{pmatrix}
\]
with $\be\neq 0$. \\[0.2cm]
4. $\Ln$ is of type $A$ with $r = \left(0,0,r_3,0,0,0,0,0,0\right),\;r_3\neq 0$ and 
\[
L(e_1)=\begin{pmatrix} 0 &  0 & 0 \\[0.2cm]
0 & 0 & 0 \\[0.2cm]
\al & \be & 0
\end{pmatrix},\quad 
L(e_2)=\begin{pmatrix} 0 & \ga  & 0 \\[0.2cm]
0 & 0 & 0 \\[0.2cm]
r_3-1+\be & \de  & 0
\end{pmatrix}
\]
with $\al\ga=0$. \\[0.2cm]
5. $\Ln$ is of type $B$ with $r = \left(0,0,0,r_4,r_5,0,-\frac{r_4^2}{r_5},-r_4,0\right),\;r_5\neq 0$ and 
\[
L(e_1)=\begin{pmatrix} \frac{r_4\al}{r_5} & -\frac{r_4^2\al}{r_5^2}  & -\frac{r_4}{2}  \\[0.2cm]
\al & -\frac{r_4\al}{r_5} & -\frac{r_5}{2} \\[0.2cm]
\be & -\frac{r_4\be}{r_5} & 0
\end{pmatrix},\quad 
L(e_2)=\begin{pmatrix} -\frac{r_4^2\al}{r_5^2} & \frac{r_4^3\al}{r_5^3}  & \frac{r_4^2}{2r_5} \\[0.2cm]
-\frac{r_4\al}{r_5} & \frac{r_4^2\al}{r_5^2} &  \frac{r_4}{2}\\[0.2cm]
-\frac{r_4\be+r_5}{r_5} & \frac{r_4(r_4\be+r_5)}{r_5^2}  & 0
\end{pmatrix}
\]
6. $\Ln$ is of type $C$ with $r = \left(0,0,0,0,0,0,r_7,0,0\right),\;r_7\neq 0$ and 
\[
L(e_1)=\begin{pmatrix} 0 &  0 & 0 \\[0.2cm]
0 & 0 & 0 \\[0.2cm]
0 & 1 & 0
\end{pmatrix},\quad 
L(e_2)=\begin{pmatrix} 0 & \al  & -\frac{r_7}{2} \\[0.2cm]
0 & 0 & 0 \\[0.2cm]
0 & \be  & 0
\end{pmatrix}
\]
\end{prop}

Note that the list includes all CPA-structures on $\Ln_3(K)$ among the types $2,3,4$. This recovers
the classification given in \cite{BU51}, Proposition $6.3$.

\begin{cor}\label{5.3}
Let $x\cdot y$ be a PA-structure on $(\Lg,\Ln)$ with $\Lg\cong \Ln\cong \Ln_3(K)$. Then all
left multiplication operators $L(x)$ are nilpotent, and the following identities hold:
\begin{align*}
x\cdot \{y,z\} & = 0,\\
[x,y]\cdot z & = z\cdot [x,y],\\
[x,y\cdot z]+[x,z\cdot y] & = [y,x\cdot z]+[y,z\cdot x]
\end{align*}
for all $x,y,z\in V$. In particular
\[
x\circ y=\frac{1}{2}(x\cdot y+y\cdot x)
\]
defines a CPA-structure on $\Lg$.
\end{cor}

\begin{proof}
The identities follow from the explicit classification. It is obvious from Proposition $\ref{5.2}$
that all PA-structures satisfy
\[
L([x,y])+\frac{1}{2}\Ad([x,y])=0,
\]
which then by \eqref{op1}, applied to $[x,y]$, gives
\[
L([x,y])+R([x,y])=0.
\]
This says that $[x,y]\cdot z  = z\cdot [x,y]$ for all $x,y,z$. Then identity \eqref{15}, and hence 
\eqref{12} is satisfied, and we obtain a CPA-structure on $\Lg$ by Proposition $\ref{4.3}$.
\end{proof}

\begin{rem}
For a PA-structure on $(\Lg,\Ln)$ with $\Lg\cong \Ln\cong \Ln_3(K)$, the right
multiplications $R(x)$ need not be nilpotent for all $x\in V$. For the PA-structure of type $C$ in Proposition
$\ref{5.2}$ we have
\[
R(e_2)=\begin{pmatrix} 0 & \al  & \frac{r_7}{2} \\[0.2cm]
0 & 0 & 0 \\[0.2cm]
1 & \be  & 0
\end{pmatrix},
\]
which has characteristic polynomial $t^3-\frac{r_7}{2}t$ with $r_7\neq 0$.
\end{rem}

\section*{Acknowledgments}
Dietrich Burde is supported by the Austrian Science Foun\-da\-tion FWF, grant P28079 
and grant I3248. Christof Ender is supported by the Austrian Science Foun\-da\-tion FWF, 
grant P28079. Wolfgang A. Moens acknowledges support by the Austrian Science Foun\-da\-tion FWF, 
grant P30842.


\begin{thebibliography}{99}

\bibitem{BAI} C. Bai, L. Guo, J. Pei: {\it  Rota-Baxter operators on $\Ls\Ll(2,\C)$ and solutions of the classical 
Yang-Baxter equation}. J.\ Math.\ Phys.\ \textbf{55} (2014), no. 2, 021701, 17 pp.

\bibitem{BU5} D. Burde: {\it Affine structures on nilmanifolds}. International Journal of
Mathematics \textbf{7} (1996), no. 5, 599--616.

\bibitem{BU19} D. Burde, K. Dekimpe, S. Deschamps: {\it The Auslander conjecture for NIL-affine
crystallographic groups}. Mathematische Annalen \textbf{332} (2005), no. 1, 161--176.

\bibitem{BU24} D. Burde: {\it Left-symmetric algebras, or pre-Lie algebras in geometry
and physics}. Central European Journal of Mathematics \textbf{4} (2006), no. 3, 323--357.

\bibitem{BU33} D. Burde, K. Dekimpe and S. Deschamps: {\it Affine actions on nilpotent
Lie groups}. Forum Math.\ \textbf{21} (2009), no. 5, 921--934.

\bibitem{BU34} D. Burde, K. Dekimpe and S. Deschamps: {\it LR-algebras}.
Contemporary Mathematics \textbf{491} (2009), 125--140.

\bibitem{BU38} D. Burde, K. Dekimpe, K. Vercammen: {\it Complete LR-structures on solvable
Lie algebras}. Journal of Group Theory \textbf{13} (2010), no. $5$, 703--719.

\bibitem{BU41} D. Burde, K. Dekimpe and K. Vercammen: {\it Affine actions on Lie groups
and post-Lie algebra structures}.
Linear Algebra and its Applications \textbf{437} (2012), no. 5, 1250--1263.

\bibitem{BU44} D. Burde, K. Dekimpe: {\it Post-Lie algebra structures and generalized
derivations of semisimple Lie algebras}.
Moscow Mathematical Journal, Vol. \textbf{13} (2013), Issue 1, 1--18.

\bibitem{BU51} D. Burde, K. Dekimpe: {\it Post-Lie algebra structures on pairs of Lie algebras}.
Journal of Algebra \textbf{464}(2016), 226--245.

\bibitem{BU52} D. Burde, W. A. Moens: {\it Commutative post-Lie algebra structures on Lie algebras}.
Journal of Algebra \textbf{467} (2016), 183--201.

\bibitem{BU57} D. Burde, W. A. Moens, K. Dekimpe: {\it Commutative post-Lie algebra structures and
linear equations for nilpotent Lie algebras}.
arXiv:1711.01964 (2017), 1--14.

\bibitem{ELM} K. Ebrahimi-Fard, A. Lundervold, I. Mencattini, H. Z. Munthe-Kaas: {\it Post-Lie Algebras and 
Isospectral Flows}.
SIGMA Symmetry Integrability Geom.\ Methods Appl.\ \textbf{11} (2015), Paper 093, 16 pp.

\bibitem{GUB} V. Gubarev: {\it Universal enveloping Lie Rota-Baxter algebra of pre-Lie and post-Lie algebras}. \\
arXiv:1708.06747 (2017), 1--13.

\bibitem{HEL} J. Helmstetter: {\it Radical d'une alg\`ebre sym\'etrique
a gauche}. Ann.\ Inst.\ Fourier \textbf{29} (1979), 17--35.

\bibitem{KIM} H. Kim: {\it Complete left-invariant affine structures on nilpotent Lie groups}.
J. Differential Geom.  \textbf{24} (1986), no. 3, 373--394.

\bibitem{LOD} J.-L. Loday: {\it Generalized bialgebras and triples of operads}.
Astérisque  No. \textbf{320} (2008), 116 pp. 

\bibitem{MED} A. Medina: {\it Flat left-invariant connections adapted to the automorphism structure of a Lie group}.
J.\ Differential Geom.\ \textbf{16} (1981), no. 3, 445--474.

\bibitem{MOE1} W. A. Moens: {\it Arithmetically-free group-gradings of Lie algebras II}.
J.\ Algebra \textbf{492} (2017), 457--474. 

\bibitem{MOE2}  W. A. Moens: {\it Arithmetically-free group-gradings of Lie algebras}.
 arXiv:1604.03459 (2016), 1--22.

\bibitem{SEG} D. Segal: {\it The structure of complete left-symmetric algebras}.
Math.\ Ann.\ \textbf{293} (1992), 569--578.

\bibitem{VAL} B. Vallette: {\it Homology of generalized partition posets}.
J.\ Pure and Applied Algebra \textbf{208} (2007), no. 2, 699--725. 


\end{thebibliography}
\end{document}